\documentclass[11pt,reqno,letterpaper]{amsart}
\usepackage{color}
\usepackage[colorlinks=true, allcolors=blue,backref=page]{hyperref}
\usepackage{amsmath, amssymb, amsthm}
\usepackage{mathrsfs}
\usepackage{mathtools}
\usepackage[noabbrev,capitalize,nameinlink]{cleveref}
\crefname{equation}{}{}
\usepackage{fullpage}
\usepackage[noadjust]{cite}
\usepackage{graphics}
\usepackage{pifont}
\usepackage{tikz}
\usepackage{bbm}
\usepackage[T1]{fontenc}

\usetikzlibrary{arrows.meta}

\usepackage{environ}
\usepackage{framed}
\usepackage{url}
\usepackage[linesnumbered,ruled,vlined]{algorithm2e}
\usepackage[noend]{algpseudocode}
\usepackage[labelfont=bf]{caption}
\usepackage{cite}
\usepackage{framed}
\usepackage[framemethod=tikz]{mdframed}
\usepackage{appendix}
\usepackage{graphicx}
\usepackage[textsize=tiny]{todonotes}
\usepackage{tcolorbox}
\usepackage{enumerate}
\allowdisplaybreaks[1]
\usepackage{enumerate}
\usepackage{stmaryrd}
\usepackage[margin=1in]{geometry}

\usepackage[shortlabels]{enumitem}
\crefformat{enumi}{#2#1#3}
\crefrangeformat{enumi}{#3#1#4 to~#5#2#6}
\crefmultiformat{enumi}{#2#1#3}
{ and~#2#1#3}{, #2#1#3}{ and~#2#1#3}

\DeclareSymbolFont{symbolsC}{U}{pxsyc}{m}{n}
\SetSymbolFont{symbolsC}{bold}{U}{pxsyc}{bx}{n}
\DeclareFontSubstitution{U}{pxsyc}{m}{n}
\DeclareMathSymbol{\medcircle}{\mathbin}{symbolsC}{7}

\crefname{algocf}{Algorithm}{Algorithms}

\crefname{equation}{}{} 
\AtBeginEnvironment{appendices}{\crefalias{section}{appendix}} 

\usepackage[color,final]{showkeys} 

\colorlet{refkey}{orange!20}
\colorlet{labelkey}{blue!30}

\crefname{algocf}{Algorithm}{Algorithms}

\numberwithin{equation}{section}
\newtheorem{theorem}{Theorem}[section]

\newtheorem{lemma}[theorem]{Lemma}

\crefname{claim}{Claim}{Claims}

\newtheorem*{question*}{Question}
\newtheorem{fact}[theorem]{Fact}

\theoremstyle{definition}
\newtheorem{definition}[theorem]{Definition}

\newtheorem*{definition*}{Definition}

\theoremstyle{remark}
\newtheorem*{remark}{Remark}

\newcommand{\mb}{\mathbb}

\newcommand{\mbm}{\mathbbm}
\newcommand{\mc}{\mathcal}

\newcommand{\mr}{\mathrm}

\newcommand{\on}{\operatorname}

\let\originalleft\left
\let\originalright\right
\renewcommand{\left}{\mathopen{}\mathclose\bgroup\originalleft}
\renewcommand{\right}{\aftergroup\egroup\originalright}

\allowdisplaybreaks

\newif\ifpublic

\ifpublic

\newcommand{\ignore}[1]{}

\else

\fi

\title{Enumerating coprime permutations}

\author[A1]{Ashwin Sah}
\author[A2]{Mehtaab Sawhney}
\address{Department of Mathematics, Massachusetts Institute of Technology, Cambridge, MA 02139, USA}
\email{\{asah,msawhney\}@mit.edu}

\thanks{Sah and Sawhney were supported by NSF Graduate Research Fellowship Program DGE-1745302. Sah was supported by the PD Soros Fellowship.}

\begin{document}

\maketitle
\begin{abstract}
Define a permutation $\sigma$ to be coprime if $\gcd(m,\sigma(m)) = 1$ for $m\in[n]$. In this note, proving a recent conjecture of Pomerance, we prove that the number of coprime permutations on $[n]$ is $n!\cdot (c+o(1))^n$ where
\[c = \prod_{p\text{ prime }}\frac{(p-1)^{2(1-1/p)}}{p\cdot (p-2)^{(1-2/p)}}.\]
The techniques involve entropy maximization for the upper bound, and a mixture of number-theoretic bounds, permanent estimates, and the absorbing method for the lower bound.
\end{abstract}

\section{Introduction}\label{sec:introduction}

\begin{definition}\label{def:coprime-permutation}
Define a permutation $\sigma\colon[n]\to [n]$ to be \emph{coprime} if $\on{gcd}(m,\pi(m)) = 1$ for $m\in[n]$. Let $C(n)$ be the number of coprime permutations on $[n]$.
\end{definition}

In recent work, Pomerance \cite{Pom22} provided nontrivial bounds on $C(n)$ proving that $n!/3.73^n<C(n)<n!/2.5^n$ for $n$ sufficiently large. Pomerance then conjectured the existence of a constant $c$ such that $C(n) = n! (c+o(1))^n$ (see \cite[Section~5]{Pom22}); our main objective is to provide such an asymptotic.

\begin{theorem}\label{thm:main}
\[C(n) = n!\bigg(\prod_{p\emph{ prime }}\frac{(p-1)^{2(1-1/p)}}{p\cdot (p-2)^{(1-2/p)}} + \exp(-\Omega(\sqrt{\log n\log\log n}))\bigg)^n.\]
\end{theorem}
\begin{remark}
The $p = 2$ term is assumed to be $1/2$. The same method demonstrates that for fixed $k$, $C_k(n)$, which counts permutations $\sigma\colon[n]\to[n]$ with $\gcd(\ell,\sigma(\ell),k!) = 1$ (see \cite[Section~5]{Pom22}), satisfies
\[C_k(n) = n!\bigg(\prod_{p\le k}\frac{(p-1)^{2(1-1/p)}}{p\cdot (p-2)^{(1-2/p)}}\bigg)^n\cdot\exp(O(\log n)),\]
though we forgo the details. McNew \cite{McN22} has independently achieved similar results for fixed $k$.
\end{remark}

\cref{thm:main} is certainly not the only result regarding finding mappings between sets of integers such that the labels of the matched edges are coprime. One such example is work of Pomerance and Selfridge \cite{PS80}, answering a question of Newman, demonstrating that such a coprime matching exists between any set of $n$ consecutive integers $A$ and $[n]$. Another is work of Bohman and Peng \cite{BP21}, in connection with the lonely runner conjecture, proving that any two intervals $A,B$ of length $2m$ in $[n]$ with $m = 2^{\Omega((\log\log n)^2)}$ have a coprime matching. This was strengthened in later work of Pomerance \cite{Pom21} to require only $m = \Omega((\log n)^2)$. We also mention that enumerating $C(n)$ was also considered by Jackson \cite{Jac77} and is listed as A005326 on the OEIS \cite{OEIS}.

\subsection{Proof techniques}\label{sub:techniques}
For the upper bound we study $C(n)$ by looking at permutations where $\gcd(j,\pi(j))$ has no prime divisor at most a growing function $W$. Defining $\tau = \prod_{p\le W}p$ and examining the portion of integers $j$ which are $a\bmod\tau$ and map to $b\bmod\tau$, we find that the problem essentially immediately reduces to an \emph{entropy maximization} program. Using the concavity and subadditivity of entropy as well as the Chinese Remainder Theorem, we can obtain the upper bound, which is done in \cref{sec:upper-bound}. Note that as written, this approach essentially necessarily requires $W = O(\log n)$ as otherwise $\tau>n$ and the integers $\{1,\ldots,n\}$ are not equidistributed in residue classes $\bmod\tau$. In order to do better, we (a) restrict attention only to whether any prime $p\le W$ divides or does not divide an integer $j\in[n]$, and (b) focus only on possible prime divisor patterns which produce large buckets (specifically, buckets which have at most $k\approx(\log n)^{1/2}$ prime divisors), and treat the rest as an error term. Since $\prod_{p\le W}(1-p^{-1}) = \Theta(1/\log W)$ is a negligible factor, the condition on $W$ we obtain is roughly that $W^k\le n^{1/100}$, which allows for consideration of significantly more prime divisors.

For the proof of the lower bound in \cref{sec:lower-bound}, the role of larger prime factors immediately becomes crucial. Let $B_S$ be the set of integers in $[n]$ with $S$ being the precise set of prime divisors at most $W$. For the sake of this discussion $W = \exp(c(\log n)^{1/2}(\log\log n)^{1/2})$ and $k = (\log n)^{1/2}(\log\log n)^{-1/2}$. Our proof is based on first taking a ``template'', assigning the appropriate ``entropy-maximizing'' fraction of $B_S$ to map to $B_{S'}$. Noting the size of $W$ it follows immediately that (say) \[\sum_{\substack{p|n\\p>W}}1/p \le W^{-1/2}\]
as $n$ has at most $2\log n/(\log\log n)$ distinct prime divisors. This suggests that the effect of ``large primes'' is minimal and therefore the upper bound proof should be reversible. 

To make this argument rigorous, one chooses random subsets of $B_S$ to map to $B_{S'}$ in accordance with the entropy maximizing distribution. One then proves that with high probability the associated coprimality graph corresponding to $B_S,B_{S'}$, if $S$ and $S'$ have size at most $k$, has at most approximately a $W^{-1/2}$ fraction of edges missing at any vertex. This is done via the Chernoff bound and the prior estimate regarding large prime factors. Combined with estimates regarding the number of perfect matchings in graph of high minimum degree due to Alon, R\"odl, and Ruci\'nski \cite{ARR98}, one can (essentially) handle matching all integers with at most $k$ prime factors at most $W$.

For the remaining integers, we ``absorb'' this small subset into the set of integers divisible by no primes at most $W$, which is a very large set, namely of size $\Omega(n/\log W)$. On the other hand, the number of integers with at least $k$ prime factors less than $W$ is $n\exp(-\Omega(k\log k))$. See \cite{Sze13} for a more extensive discussion of the absorbing method and its history. Our approach here can be thought of as analogous to a ``robust absorption strategy'', which has appeared for instance in work of Montgomery \cite{Mon19}, in which bipartite graphs with a ``resilient matching'' property are utilized. Finally, we note that we have omitted certain special considerations at the prime $p=2$, which are handled via a reduction of Pomerance \cite{Pom22} (\cref{lem:even-reduction}).

\subsection{Further directions}\label{sub:further}
We end by mentioning that we believe obtaining sharper asymptotics for $C(n)$ is an interesting open problem; our methods are fundamentally limited to error terms that are of the form $n^{-o(1)}$ as we crucially rely on considering buckets whose shared prime divisors multiply to less than $n^{1/2}$ (say) in order to guarantee that the interval has the expected number of elements in such buckets. Thus either we cannot consider many primes, or cannot consider buckets with many prime divisors, either of which is a serious issue. Given this, error terms of the form $O(n^{-\Omega(1)})$ would require a substantial modification of our methods and even the truth of such a statement is not obvious. Additionally, obtaining an asymptotic expansion for $C(n)$ appears unlikely (as speculated by Jackson \cite{Jac77}).

We note that a reduction of Pomerance \cite{Pom22} shows it is enough to count coprime bijections $f\colon\{1,3,\ldots,2n-1\}\to[n]$. This is the number of matchings in a bipartite graph $i\in\{1,\ldots,2n-1\}$ on the left and $j\in[n]$ on the right are connected if $\gcd(i,j) = 1$. Our proof can be interpreted as considering the non-edges that occur due to small primes and noting that the ``remaining non-edges'' are sparse. For sharper bounds, it may be possible to analyze an edge-removal process that targets the maximum entropy distribution described in \cref{sub:techniques}: at each step, uniformly randomly choose an edge $(i,j)$ in the current graph with probability say proportional to some $g(i,j)$ and then delete $i$ from the left and $j$ from the right. Here we might take
\[g(i,j) = \prod_{2<p\le n}g_p(\mbm{1}_{p|i},\mbm{1}_{p|j})\]
with $g_p(0,0) = (p-2)/p$, $g_p(0,1) = g_p(1,0) = 1/p$, and $g_p(1,1) = 0$. Depending on how long one can control this process, the remaining vertices can be handled by the absorbing method. Our proof can be thought of as very similar to analyzing such a process where we truncate the weights at $p\le W$.

Techniques involving ``weighted removal processes'' or that target an ``entropy-maximizing distribution'' along with the absorbing method have been recently used to give sharp estimates on the $n$-queens problem \cite{Sim21}, and may be applicable to further situations where one wishes to count objects with structured biases (such as number-theoretic considerations, as in the present work). This in turn builds on recent powerful techniques introduced for counting in quasirandom situations such as counting combinatorial designs \cite{Kee18}, in which polynomially strong error terms are known. It remains interesting to see the limits of such techniques in number-theoretic settings.

\subsection*{Notation}
Throughout this paper we will reserve $p$ for variables which range over the primes. We will also implicitly assume $n$ is sufficiently large for certain inequalities to hold. We write $a = b\pm c$ to mean that $a\in[b-c,b+c]$, and $f = O(g)$ means $|f|\le Cg$ for some absolute constant $C$ while $f = \Omega(g)$ means $f\ge C|g|$ for some absolute $C$. We use $\log^{(k)}n$ to denote the $k$-times iterated logarithm, e.g. $\log^{(5)}n = \log\log\log\log\log n$. We may sometimes omit floors and ceilings without comment if not essential to the situation. We say an event occurs whp (with high probability) if it occurs with probability tending to $1$ as $n\to\infty$. Given a set $S$, we write $\binom{S}{k},\binom{S}{\le k},\binom{S}{< k}$ for the collection of subsets of $S$ of size $k$, size at most $k$, and size less than $k$, respectively. Finally, as above, we often refer to a bijective function or permutation as a matching between the domain and codomain, especially when dealing with extra constraints encoded via a graph.

\subsection*{Acknowledgements}
We thank Carl Pomerance for comments motivating our proof of \cref{lem:bound-size} and for communicating the results of McNew \cite{McN22} regarding $C_k(n)$.

\section{Upper Bound}\label{sec:upper-bound}
\subsection{Setup}\label{sub:upper-setup}
Fix $W = \exp(2^{-10}(\log n\log\log n)^{1/2})$ and $k =  2^{-5}(\log n/\log\log n)^{1/2}$. Let $\alpha = \exp(-k\log k)$, which will roughly correspond to the error term we derive. Our upper bound on the number of coprime permutations $\sigma\colon[n]\to[n]$ comes via enumerating permutations such that $p\nmid\gcd(j,\sigma(j))$ for all $p\le W$. Let $\mc{P}_W = \{p\colon p\le W\}$ be the primes which are at most $W$. 

We first define a series of sets with respect to the divisibility structure of $\mc{P}_W$. For each $S\subseteq\mc{P}_W$, let 
\[A_S = [n] \cap \bigcap_{p\in S}\{m\colon p|m\}\cap\bigcap_{p\in  \mc{P}_W\setminus S}\{m\colon p\nmid m\}\]
and 
\[A_{\ge k} = [n]\setminus\bigg(\bigcup_{S\in \binom{\mc{P}_W}{<k}}A_S\bigg).\]

To start, we prove that $A_S$ has the ``predicted size'' from an independent heuristic on small primes if $|S| < k$, while $A_{\ge k}$ is a vanishing small proportion of the set. This will allow us to efficiently encode the divisibility structure of all primes at most $W$. Note that this is nontrivial as a more naive bound based on the Chinese Remainder Theorem (CRT) would restrict us to having $W$ such that $\prod_{p\le W}p\le n^{1/2}$ (say) and hence $W\le\log n$. 

\begin{lemma}\label{lem:bound-size}
For all $S\subseteq \mc{P}_W$ with $|S|<k$ we have
\[|A_S| = (1\pm\alpha^2)n\prod_{p\in S}p^{-1}\prod_{p\in\mc{P}_W\setminus S}(1-p^{-1}).\]
Furthermore,
\[|A_{\ge k}|\le n\alpha^{1/2}.\]
\end{lemma}
\begin{proof}
We first handle $A_{\ge k}$. Let $Y$ be a uniformly random positive integer in $[n]$ and $Y_p$ be the indicator of $p|Y$ for $1\le p\le W$. Note that 
\begin{align}
\mb{P}[Y\in A_{\ge k}]&\le \sum_{1\le p_1<\cdots<p_k\le W}\mb{P}\bigg[\bigcap_{1\le j\le k}Y_{p_j}\bigg]=\sum_{1\le p_1<\cdots<p_k\le W}\mb{P}\bigg[\bigwedge_{1\le j\le k}(p_j|Y)\bigg]\notag\\
&\le 2\sum_{1\le p_1<\cdots<p_k\le W}\prod_{1\le j\le k}\frac{1}{p_j}\le \frac{2}{k!}\bigg(\sum_{1\le p\le W}\frac{1}{p}\bigg)^{k}\le\exp(-k\log k/2)\label{eq:small-set-bound}
\end{align}
where we have used that $W^k\le n$ to derive the probability that $Y$ is divisible by $\prod_{1\le j\le k}p_j$ and used that $\sum_{1\le p\le W}1/p = O(\log\log n)$. To derive the necessary bound on $A_S$, we proceed via the Principle of Inclusion-Exclusion (PIE), in particular with the use of Bonferroni's inequalities \cite{Bon36}. In particular, we have
\begin{align*}
\mb{P}[Y\in A_S]& = \mb{E}\bigg[\prod_{p\in S}Y_{p}\prod_{p\in \mc{P}_W\setminus S}(1-Y_{p})\bigg]\\
&=\sum_{0\le j\le 4k}(-1)^{j}\sum_{T\subseteq \binom{\mc{P}_W\setminus S}{j}}\mb{E}\bigg[\prod_{p\in S}Y_{p}\prod_{p\in T}Y_p\bigg] \pm \sum_{T\subseteq \binom{\mc{P}_W\setminus S}{4k}}\mb{E}\bigg[\prod_{p\in S}Y_{p}\prod_{p\in T}Y_{p}\bigg]\\
&= \sum_{0\le j\le 4k}(-1)^{j}\sum_{T\subseteq \binom{\mc{P}_W\setminus S}{j}}\prod_{p\in S\cup T}p^{-1}\pm \bigg(\bigg(\prod_{p\in S}p^{-1}\bigg)\exp(-3k\log k) + (2W)^{4k}/n\bigg)\\
&= \prod_{p\in S}p^{-1}\prod_{p\in\mc{P}_W\setminus S}(1-p^{-1})\pm \bigg(\bigg(\prod_{p\in S}p^{-1}\bigg)\exp(-3k\log k) + (2W)^{4k}/n\bigg)\\\\
& =  (1\pm \exp(-2k\log k))\prod_{p\in S}p^{-1}\prod_{p\in \mc{P}_W\setminus S}(1-p^{-1}).
\end{align*}
The first error term in the third line comes from a bound similar to \cref{eq:small-set-bound}, while the second error term in the third line comes from the slight deviations in probability that a $Y$ is divisible by some product of primes at most $W^k\le n$. The fourth line comes from bounding the tail terms of the product expansion of $\prod_{p\in S}p^{-1}\prod_{p\in\mc{P}_W\setminus S}(1-p^{-1})$ (e.g.~by use of the Bonferroni inequalities once again) and the fifth line follows upon converting additive error to multiplicative error, using $W^k\le n^{1/8}$ and $\prod_{p\in\mc{P}_W}(1-p^{-1}) = \Omega(1/\log W)$. This completes the proof.
\end{proof}

Given a coprime permutation $\sigma\colon[n]\to[n]$ and $S_1,S_2\subseteq \mc{P}_W$ we let
\[\beta_{S_1,S_2}(\sigma) = \#\{m\in[n]\colon m\in A_{S_1},~\sigma(m)\in A_{S_2}\}\]
and $\beta(\sigma)$ be the tuple of all such counts. Note there are at most $(2W)^k\le\exp(k\log k)\le n^{1/4}$ such sets $S$ with $|S|<k$ and that by \cref{lem:bound-size}, we have $|A_{\ge k}|\le n\alpha^{1/2}$. This implies that there are at most
\[(n+1)^{\sqrt{n}}\cdot(2n^2)^{n\alpha^{1/2}}\le e^{n\alpha^{1/3}}\]
possible choices for $\beta(\sigma)$: first, for every pair $S_1,S_2$ with $|S_1|,|S_2| < k$ we choose up to $n+1$ possibilities, and then we bound the ways to simultaneously choose all $\beta_{S,S'}(\sigma)$ for $S\subseteq\mc{P}_W$ and $|S'|\ge k$ (and similar for $\beta_{S',S}(\sigma)$), which is a total of $t\le n^2$ values which must sum to $|A_{\ge k}|$.

Therefore we have that
\begin{equation}\label{eq:upper-single}
C(n)\le e^{n\alpha^{1/3}}\cdot \sup_{\beta}\{\sigma\colon[n]\to[n]\text{ coprime}\colon\beta(\sigma)=\beta\}
\end{equation}

Finally we define the following parameters:
\[\rho_{S_1,S_2} = \beta_{S_1,S_2}/n,~\beta_{S_1,\cdot} = \sum_{S_2}\beta_{S_1,S_2} = |A_{S_1}|,~\beta_{\cdot,S_2} = \sum_{S_1}\beta_{S_1,S_2} = |A_{S_2}|.\]

\subsection{Completing the upper bound}\label{sub:upper-bound}
We now bound the number of choices of $\sigma$ given $\beta(\sigma)$. This will complete the proof of the upper bound due to \cref{eq:upper-single}. We bound the number of permutations $\sigma$ satisfying $\beta(\sigma) = \beta$ as follows:
\begin{itemize}
\item For each pair of subsets $S_1,S_2\subseteq \mc{P}_W$, choose which $\beta_{S_1,S_2}$ elements of $A_{S_1}$ which have image in $A_{S_2}$. Call such elements $V_{S_1\to S_2}$.
\item For each pair of subsets $S_1,S_2\subseteq \mc{P}_W$, choose which $\beta_{S_1,S_2}$ elements of $A_{S_2}$ which have preimage in $A_{S_1}$. Call such elements $V_{S_1\leftarrow S_2}$.
\item Determine a precise mapping between $V_{S_1\to S_2}$ and $V_{S_1\leftarrow S_2}$.
\end{itemize}
Note also that if $\rho_{S_1,S_2}\neq 0$ where $S_1\cap S_2\neq\emptyset$, then the expression within the supremum in \cref{eq:upper-single} is actually $0$. Thus we can restrict attention to values of $\beta$ where $\rho_{S_1,S_2} = 0$ when $S_1\cap S_2\neq\emptyset$.

Given this procedure, we see that the number of $\sigma$ such that $\beta(\sigma) = \beta$ is bounded by 
\begin{align}
&\prod_{S_1,S_2\subseteq \mc{P}_W}\beta_{S_1,S_2}!\prod_{S_1\subseteq \mc{P}_W}\frac{\beta_{S_1,\cdot}!}{\prod_{S_2\subseteq \mc{P}_W}\beta_{S_1,S_2}!}\prod_{S_2\subseteq \mc{P}_W}\frac{\beta_{\cdot,S_2}!}{\prod_{S_1\subseteq \mc{P}_W}\beta_{S_1,S_2}!}\notag\\
&=\prod_{S_1,S_2\subseteq \mc{P}_W}(\beta_{S_1,S_2}!)^{-1}\prod_{S_1\subseteq \mc{P}_W}\beta_{S_1,\cdot}!\prod_{S_2\subseteq \mc{P}_W}\beta_{\cdot,S_2}!\notag\\
&\le |A_{\ge k}|!^2\prod_{|S_1|,|S_2|< k}(\beta_{S_1,S_2}!)^{-1}\prod_{|S_1|< k}\beta_{S_1,\cdot}!\prod_{|S_2|< k}\beta_{\cdot,S_2}!\notag\\
&\le e^{n\alpha^{1/3}}\cdot\prod_{|S_1|,|S_2|< k}(\beta_{S_1,S_2}!)^{-1}\prod_{|S|< k}|A_S|!^2\label{eq:factorial-upper}
\end{align}
where we used that factorials are log-convex in the third line to combine the terms for $|S|\ge k$, and used \cref{lem:bound-size} and Stirling's approximation in the fourth.

We now bound each term in the above product. Write $n_S = n\prod_{p\in S}p^{-1}\prod_{p\in\mc{P}_W\setminus S}(1-p^{-1})$. Noting that $n!\le 3(n+1)(n/e)^n$ for all $n\ge 1$ and $(2W)^k\le n^{1/4}$, and applying \cref{lem:bound-size} we find that
\begin{align*}
\prod_{|S|\le k}|A_S|!&\le e^{n\alpha^{1/3}}\cdot \prod_{|S|<k}((1+\alpha^2)n_S/e)^{(1+\alpha^2)n_S}\\
&\le e^{2n\alpha^{1/3}}\cdot\prod_{S\subseteq \mc{P}_W}\bigg((n/e)\prod_{p\in S}p^{-1}\prod_{p\in \mc{P}_W\setminus S}(1-p^{-1})\bigg)^{n\prod_{p\in S}p^{-1}\prod_{p\in \mc{P}_W\setminus S}(1-p^{-1})}\\
&= e^{2n\alpha^{1/3}}\cdot (n/e)^n\prod_{p\in \mc{P}_W}(1-1/p)^{n(1-1/p)}(1/p)^{n/p}
\end{align*}

Next noting that $n!\ge (n/e)^n$ for all $n\ge 1$, we find that
\begin{align*}
\prod_{|S_1|,|S_2|< k}(\beta_{S_1,S_2}!)^{-1}&\le e^n\prod_{|S_1|,|S_2|< k}{\beta_{S_1,S_2}}^{-\beta_{S_1,S_2}}\le e^{n\alpha^{1/3}}\cdot e^n\prod_{S_1,S_2\subseteq \mc{P}_W}{\beta_{S_1,S_2}}^{-\beta_{S_1,S_2}}\\
&= e^{n\alpha^{1/3}}\cdot e^n\exp\Big(-n\sum_{S_1,S_2}\rho_{S_1,S_2}\log\beta_{S_1,S_2}\Big)\\
&= e^{n\alpha^{1/3}}\cdot (e/n)^n\exp\Big(-n\sum_{S_1,S_2}\rho_{S_1,S_2}\log\rho_{S_1,S_2}\Big).
\end{align*}

Let $Z$ be the random variable which takes on the value $(S_1,S_2)$ with probability $\rho_{S_1,S_2}$. Treating each set $S_i$ as a vector in $\{0,1\}^{\mc{P}_W}$, we may equivalently view $Z = ((\mbm{1}_{p\in S_1}, \mbm{1}_{p\in S_2}))_{p\in \mc{P}_W}\in\{0,1\}^{2|\mc{P}_W|}$. Let $Z_p$ be $(\mbm{1}_{p\in S_1}, \mbm{1}_{p\in S_2})$ when $(S_1,S_2)$ is drawn according to $Z$, and note that $Z$ is determined upon knowing $Z_p$ for all $p\in\mc{P}_W$. By the subadditivity of entropy we therefore find that 
\[\exp\Big(-n\sum_{S_1,S_2}\rho_{S_1,S_2}\log\rho_{S_1,S_2}\Big)=\exp(nH(Z))\le\prod_{p\in\mc{P}_W}\exp(nH(Z_p)).\]
Note that $Z_p$ is supported on $(0,0), (1,0), (0,1)$ due to the constraint $\rho_{S_1,S_2} = 0$ whenever we have $p\in S_1\cap S_2$. Additionally, $\mb{P}[Z_p = (1,\cdot)] = \mb{P}[Z_p = (\cdot,1)] = \lfloor n/p\rfloor/n$ since $\sigma$ is a permutation and $|\{p|\ell\}\cap[n]| = \lfloor n/p\rfloor$. These constraints determine the law of $Z_p$ and therefore $H(Z_p)$. In particular, it follows that $\mb{P}[Z_p = (1,0)] = \mb{P}[Z_p = (0,1)] = 1/p\pm 1/n$ and $\mb{P}[Z_p = (0,0)] = 1-2/p \pm 2/n$.

Putting everything together into \cref{eq:factorial-upper}, this gives an upper bound of
\begin{align*}
&e^{7n\alpha^{1/3}}\cdot (n/e)^n\bigg(\prod_{p\in \mc{P}_W}(1-1/p)^{n(1-1/p)}(1/p)^{n/p}\bigg)^2\prod_{p\in \mc{P}_W}(1-2/p)^{-n(1-2/p)}(1/p)^{-2n/p}\\
&\qquad= e^{7n\alpha^{1/3}}\cdot (n/e)^{n} \prod_{p\in \mc{P}_W}(p-1)^{2(1-1/p)n}(p-2)^{-n(1-2/p)}p^{-n}\\
&\qquad\le e^{7n\alpha^{1/3} + n/W}\cdot (n/e)^{n} \prod_{p}(p-1)^{2(1-1/p)n}(p-2)^{-n(1-2/p)}p^{-n}
\end{align*}
where we have used that $(p-1)^{2(1-1/p)}(p-2)^{-(1-2/p)}p^{-1} = 1+O(1/p^2)$ which implies the inequality $\prod_{p>W}(p-1)^{2(1-1/p)}(p-2)^{-(1-2/p)}p^{-1} =  1 + O(1/(W\log W))\le\exp(1/W)$. The desired result follows immediately using our expressions for $\alpha, W$.

\section{Lower Bound}\label{sec:lower-bound}
\subsection{Setup and preliminaries}\label{sub:lower-setup}
Recall that we have set $W = \exp(2^{-10}(\log n\log\log n)^{1/2})$ and $k =  2^{-5}(\log n/\log\log n)^{1/2}$, as well as $\alpha = \exp(-k\log k)$.

For the lower bound, note that any coprime permutation must map even numbers to odd numbers, which are sets of approximately the same size. In order to witness this phenomenon, we use the following reduction of Pomerance \cite{Pom22}. Define $[n]_\mr{o} = \{1,3,\ldots,2n-1\}$, the first $n$ odd positive integers.
\begin{lemma}[{\cite[Lemma~1]{Pom22}}]\label{lem:even-reduction}
Let $C_0(n)$ be the number of bijective functions $f\colon[n]_\mr{o}\to[n]$ with $\gcd(j,f(j)) = 1$ for all $j\in\{1,3,\ldots,2n-1\}$. Then $C(2n) = C_0(n)^2$ and $C(2n+1)\ge 2C_0(n-1)^2$.
\end{lemma}

We will require slightly modified sets in order to capture the divisibility structures of $[n]$ and $[n]_\mr{o}$. Let $\mc{P}_W' = \mc{P}_W\setminus \{2\}$. For $S\subseteq \mc{P}_W'$ define 
\[B_S = [n]\cap \bigcap_{p\in S}\{m\colon p|m\}\cap \bigcap_{p\in  \mc{P}_W'\setminus S}\{m\colon p\nmid m\},~B_{\ge k} = [n]\setminus \bigcup_{S\in \binom{\mc{P}_W'}{< k}}B_S\]
and the analogous
\[C_S = [n]_\mr{o}\cap \bigcap_{p\in S}\{m\colon p|m\}\cap \bigcap_{p\in  \mc{P}_W'\setminus S}\{m\colon p\nmid m\},~C_{\ge k} = [n]_\mr{o}\setminus \bigcup_{S\in\binom{\mc{P}_W'}{< k}}C_S.\]
Similar to \cref{lem:bound-size}, we have the following control on the set sizes $B_S$, $C_S$.
\begin{lemma}\label{lem:bound-size-2}
For all $S\subseteq \mc{P}_W'$ with $|S|<k$ we have, 
\[|B_S|, |C_S| = (1\pm\alpha^2)n\prod_{p\in S}p^{-1}\prod_{p\in \mc{P}_W'\setminus S}(1-p^{-1}).\]
Furthermore,
\[|B_{\ge k}|,|C_{\ge k}|\le\alpha^{1/2}n.\]
\end{lemma}

For the remainder of this section we focus on providing an asymptotic lower bound on $C_0(n)$, which we treat as a matching problem between $[n]_\mr{o}$ and $[n]$ with forbidden edges corresponding to non-coprime pairs. Now, heuristically, the lower bound should match the bounding scheme presented in \cref{sec:upper-bound} (ignoring the prime $p = 2$). However, to actually ensure we avoid coprimality with respect to larger primes $p$, we will require a number of estimates. The crucial point for our lower bound is noting that the divisibility structure of primes larger than $W$ removes a sparse graph (of density bounded by say $W^{-1/2}$).

\begin{fact}\label{fact:prime-sum}
For $\ell\in [2n]$ we have that 
\[\sum_{\substack{p|\ell\\ p>W}} \frac{1}{p}<W^{-1/2}.\]
\end{fact}
\begin{proof}
This is immediate noting that at most $2\log n/(\log\log n)$ distinct primes divide an integer $\ell\in [2n]$. 
\end{proof}

We will also require that for a given ``large class'' only a small portion of numbers are divisible by a fixed $p$. 
\begin{lemma}\label{lem:prime-division}
For $W<p'\le n^{1/3}$ and $|S|< k$, we have that 
\[|B_S\cap\{m\colon p'|m\}|\le 2|A_S|/p',~|C_S\cap\{m\colon p|m\}|\le 2|B_S|/p'.\]
For $p'>n^{1/3}$ we have that
\[|B_S\cap\{m\colon p'|m\}|\le 2n^{2/3},~|C_S\cap\{m\colon p|m\}|\le 2n^{2/3}.\]
\end{lemma}
\begin{proof}
The second claim follows immediately as $|[2n]\cap  \{m: p'|m\}|\le 2n^{2/3}$. For the first claim, we focus on $B_S$ as the analogous claim for $C_S$ follows similarly. Let $Y$ and $Y_p$ be as in the proof of \cref{lem:bound-size}. Note that 
\begin{align*}
\mb{P}[Y\in A_S&\wedge Y_{p'}] = \mb{E}\bigg[Y_{p'}\prod_{p\in S}Y_{p}\prod_{p\in\mc{P}_W'\setminus S}(1-Y_{p})\bigg]\\
&=\sum_{0\le j\le 4k}(-1)^{j}\sum_{T\subseteq \binom{\mc{P}_W'\setminus S}{j}}\mb{E}\bigg[Y_{p'}\prod_{p\in S}Y_{p}\prod_{p\in T}Y_p\bigg] \pm \sum_{T\subseteq \binom{\mc{P}_W'\setminus S}{4k}}\mb{E}\bigg[Y_{p'}\prod_{p\in S}Y_{p}\prod_{p\in T}Y_{p}\bigg]\\
&= \sum_{0\le j\le 4k}(-1)^{j}\sum_{T\subseteq \binom{\mc{P}_W'\setminus S}{j}}(p')^{-1}\prod_{p\in S\cup T}p^{-1}\pm \bigg((p')^{-1}\bigg(\prod_{p\in S}p^{-1}\bigg)\exp(-3k\log k) + (2W)^{4k}/n\bigg)\\
& = (1\pm\alpha^2)(p')^{-1}\prod_{p\in S}p^{-1}\prod_{p\in\mc{P}_W'\setminus S}(1-p^{-1})
\end{align*}
where we have used that $n^{1/3}(W)^{4k}\le n^{1/4}$. The result follows immediately.
\end{proof}

We will also require that in any bipartite graph with balanced vertex parts and a very large minimum degree there are a large number of perfect matchings. The argument is identical to that of Alon, R\"odl, and Ruci\'nski \cite[Theorem~1]{ARR98}, but we reprove the result in order to quantify the dependencies.
\begin{lemma}\label{lem:matchings}
Let $G = (A\cup B, E)$ be a bipartite graph such that $|A|=|B|=n$ and $\Delta = \Delta(K_{A,B}\setminus G)$, the maximum degree of the bipartite complement, satisfies $\Delta\le n/3$. Then there are at least $((n-2\Delta)/e)^n$ perfect matchings in $G$.
\end{lemma}
\begin{proof}
By \cite[Theorem~3]{ARR98}, if for any $X\subseteq A$, $Y\subseteq B$,
\[k(|X|+|Y|) + E(A\setminus X,B\setminus Y)\ge kn\]
holds, then $G$ has a $k$-factor (and therefore a subgraph of $G$ which is $k$-regular). Let $k = n-2\Delta$ and note that for $|X|+|Y|\ge n$ the above inequality is immediate. If $|X|+|Y|\le n$, suppose that $|X|\ge |Y|$ (the other case being analogous) and note that $E(A\setminus X, B\setminus Y)\ge (n-|X|)(n-|Y|) - (n-|X|)\Delta = (n-|X|)(n-|Y|-\Delta)$ and hence 
\[k(|X|+|Y|) + E(A\setminus X,B\setminus Y)\ge k(|X|+|Y|) + (n-|X|)(n-|Y|-\Delta).\]
Therefore it suffices to check that $k(|X|+|Y|) + (n-|X|)(n-|Y|-\Delta)\ge kn$ which is equivalent to 
\[k(n-|X|-|Y|)\le (n-|X|)(n-|Y|-\Delta).\]
Fix $|X|+|Y| = t$ and imagine varying $|X|$ with $t\ge|X|\ge|Y|=t-|X|$. By concavity of the right side, it suffices to check when $|X| = t, |Y| = 0$ and $|X| = |Y| = t/2$. In the first case the inequality follows as $k\le n-\Delta$ and in the second case the inequality is equivalent to 
\[k(n-t)\le (n-t/2)(n-t/2-\Delta).\]
Viewing the difference of the right and left hand sides as functions of $t$, the derivative is seen to be $t/2 + k - n + \Delta/2$ and thus the inequality is closest when $t = 2n-2k-\Delta = 3\Delta$. It becomes
\[(n-2\Delta)(n-3\Delta)\le(n-3\Delta/2)(n-5\Delta/2)\]
which is trivially true as $n\ge 3\Delta$.

Now let $G'$ be the corresponding subgraph of $G$ which is exactly $k$-regular. Then by the solution of the van der Waerden conjecture on permanents due to Egorychev \cite{Ego81} and Falikman \cite{Fal81} it follows that the number of perfect matchings in $G'$ is at least $(k/n)^{n}n!\ge(k/e)^n$.
\end{proof}

Finally we will repeatedly require the Chernoff bound. 
\begin{lemma}[Chernoff bound]\label{lem:chernoff}
Let $X$ be either:
\begin{itemize}
    \item a sum of independent random variables, each of which takes values in $[0,1]$
    \item hypergeometrically distributed (with any parameters).
\end{itemize}
Then for any $\delta\ge 0$ we have
\[\mb{P}[X\ge(1+\delta)\mb{E}X]\le\exp(-\delta^2\mb{E}X/(2+\delta)).\]
\end{lemma}

\subsection{Proof of the lower bound}\label{sub:lower-bound}
In order to complete the proof of \cref{thm:main}, we devise a procedure that produces a triplet $(\sigma, R_B, R_C)$ where $R_B\subseteq [n]$ and $R_C\subseteq [n]_\mr{o}$ with $|R_B|, |R_C|\le n\exp(n\alpha^c)$ and such that $\sigma$ is coprime. The number of possible outcomes of such a procedure is bounded above by 
$C(n)\exp(n\alpha^{c/2})$ due to the size bounds on $R_B, R_C$ and therefore it suffices to provide a lower bound on the number possible outcomes.

Let $\mc{M}_B$ denote the set of labels $\bigcup_{\substack{S\subseteq \mc{P}_W'\\|S|<k}}\{S\}\cup\{\ast\}$ and let $\mc{M}_C$ denote the set of labels $\bigcup_{\substack{S\subseteq \mc{P}_W'\\|S|<k}}\{S\}\cup\{\ast\}$. We first construct (via a randomized procedure) a pair of maps $\phi_1\colon[n]\to\mc{M}_B$ and $\phi_2\colon[n]\to\mc{M}_C$ which will serve as a template for the procedure producing $(\sigma,R_B,R_C)$.

\begin{lemma}\label{lem:assignment}
There exist assignment functions $\phi_1\colon[n]\to\mc{M}_B$ and $\phi_2\colon[n]\to\mc{M}_C$ such that:
\begin{itemize}
\item $\phi_1(m_1) = \ast$ for $m_1\in B_{\ge k}$ and $\phi_2(m_2) = \ast$ for $m_2\in C_{\ge k}$.
\item If $S\cap T\neq\emptyset$ then $x\in B_S$ implies $\phi_1(x)\neq T$, and $y\in C_T$ implies $\phi_2(y)\neq S$.
\item For $|S_1|,|S_2|<k$, we have
\[\beta_{S_1,S_2}:=|\{m_1\in B_{S_1}\colon\phi_1(m_1)=S_2\}| = |\{m_2\in C_{S_2}\colon\phi_2(m_2)=S_1\}|.\]
\item For all $|S_1|,|S_2|<k$,
\[\beta_{S_1,S_2} = (1\pm\exp(-\Omega((\log n\log\log n)^{1/2})))n\prod_{p\in S_1}p^{-1}\prod_{p\in S_2}p^{-1}\prod_{p\in\mc{P}_W'\setminus (S_1\cup S_2)}(1-2p^{-1})\]
\item $|\{m_1\colon\phi_1(m_1) = \ast\}|\le 2\alpha^{1/2}n$.
\item $|\{m_2\colon\phi_2(m_2) = \ast\}|\le 2\alpha^{1/2}n$.
\end{itemize}
\end{lemma}
\begin{proof}
We construct initial random assignment functions which essentially only fail the third bullet point, and then ``throw out'' a small fraction of values by reassigning them to $\ast$ in order to balance the two sides. For $m_1\in B_{\ge k}$ deterministically assign $\phi_1'(m_1) = \ast$ and for $m_2\in C_{\ge k}$ deterministically assign $\phi_2'(m_2) = \ast$. Given $m_1\in B_S$ with $|S|<k$, for each $|T|<k$ we assign $\phi_1'(m_1)=T$ with probability
\[\mbm{1}_{S\cap T = \emptyset}\cdot\prod_{p\in T}(p-1)^{-1}\prod_{p\in \mc{P}_W'\setminus(S\cup T)}(1-(p-1)^{-1})\]
and otherwise set $\phi_1'(m_1) = \ast$.  Analogously, given $m_2\in C_S$ with $|S|<k$, for $|T|<k$ we assign $\phi_2'(m_2)=T$ with probability
\[\mbm{1}_{S\cap T = \emptyset}\cdot\prod_{p\in T}(p-1)^{-1}\prod_{p\in \mc{P}_W'\setminus(S\cup T)}(1-(p-1)^{-1})\]
and else set $\phi_2'(m_2) = \ast$. Evidently the stated probabilities add to at most $1$ when summing over $|T|<k$, so this is well-defined. 

By \cref{lem:bound-size-2} we have for $|S_1|,|S_2|<k$ which are disjoint that
\begin{align*}
\mb{E}&|\{m_1\in B_{S_1}\colon\phi_1'(m_1)=S_2\}|\\
&= (1\pm\alpha^2)n\prod_{p\in S_1}p^{-1}\prod_{p\in\mc{P}_{W}'\setminus S_1}(1-p^{-1})\prod_{p\in S_2} (p-1)^{-1}\prod_{p\in\mc{P}_W'\setminus (S_1\cup S_2)}(1-(p-1)^{-1})\\
&= (1\pm\alpha^2)n\prod_{p\in S_1}p^{-1}\prod_{p\in S_2}p^{-1}\prod_{p\in\mc{P}_W'\setminus (S_1\cup S_2)}(1-2p^{-1})
\end{align*}
and analogously
\[\mb{E}|\{m_2\in C_{S_2}\colon\phi_2'(m_2)=S_1\}| = (1\pm\alpha^2)n\prod_{p\in S_1}p^{-1}\prod_{p\in S_2}p^{-1}\prod_{p\in\mc{P}_W'\setminus(S_1\cup S_2)}(1-2p^{-1}).\]

Applying Chernoff (\cref{lem:chernoff}) it follows with probability at least $1/2$ that if for $|S_1|,|S_2|<k$ disjoint we set
\[\beta_{S_1,S_2} = \min(|\{m_1\in B_{S_1}\colon\phi_1'(m_1)=S_2\}|,|\{m_2\in C_{S_2}\colon\phi_2'(m_2)=S_1\}|)\]
then
\[\beta_{S_1,S_2} = (1\pm 2\alpha^2)n\prod_{p\in S_1}p^{-1}\prod_{p\in S_2}p^{-1}\prod_{p\in \mc{P}_W\setminus (S_1\cup S_2)}(1-2p^{-1}).\]

To construct $\phi_1$, for each pair of disjoint $|S_1|,|S_2|$ of size less than $k$ assign $\beta_{S_1,S_2}$ values of $S_1$ to map to the label $S_2$. Send all remaining values in $S_1$, if any, to $\ast$ (and set $B_{\ge k} = \ast$ as before). For $\phi_2$, assign $\beta_{S_1,S_2}$ values of $S_2$ to map to the label $S_1$ and send all remaining values in $S_2$, if any, to $\ast$. This is well-defined since clearly
\[\sum_{|S_2|<k}\beta_{S_1,S_2}\le\sum_{|S_2|<k}|\{m_1\in B_{S_1}\colon\phi_1'(m_1)=S_2\}|\le|B_{S_1}|\]
and similar for $\phi_2'$.

The first four properties follow immediately. The final two properties follow from noting that
\[\Big||\{m_1\in B_{S_1}\colon\phi_1'(m_1)=S_2\}|-|\{m_2\in C_{S_2}\colon\phi_2'(m_2)=S_1\}|\Big|\le 8\alpha^2\beta_{S_1,S_2},\]
so the extra amount of values assigned the label $\ast$ is not significant, and by citing the second part of \cref{lem:bound-size-2}.
\end{proof}

To construct $(\sigma, R_B, R_C)$, we define $\beta_{S_1,S_2}$ as in \cref{lem:assignment}, $\beta_\ast^1 = |\{m_1\colon\phi_1(m_1) = \ast\}|$, and $\beta_\ast^2 = |\{m_2\colon\phi_2(m_2) = \ast\}|$. Next we consider the following procedure:
\begin{itemize}
\item Fix $\phi_1$, $\phi_2$ satisfying \cref{lem:assignment} with these values of $\beta$.
\item Choose $\sigma_1$ uniformly randomly among all permutations of $[n]$ with $\sigma_1(B_S)=B_S$ for all $S$ and $\sigma_2$ uniformly randomly among all permutations of $[n]_\mr{o}$ with $\sigma_2(C_S)=C_S$ for all $S$.
\item Let $B_{S_1,S_2} = \{m_1\in B_{S_1}\colon\phi_1(\sigma_1(m_1)) = S_2\}$ and $C_{S_1,S_2} = \{m_2\in C_{S_2}\colon\phi_2(\sigma_2(m_2)) = S_1\}$ for $|S_1|,|S_2|<k$. Furthermore let $R_B = \{m_1\colon\phi_1(\sigma_1(m_1))=\ast\}$ and $R_C = \{m_2\colon\phi_2(\sigma_2(m_2)) = \ast\}$
\item For disjoint $S_1,S_2$ of size at most $k$ with $(S_1,S_2) \neq (\emptyset,\emptyset)$, let $\sigma$ form a uniformly random matching between $B_{S_1,S_2}$ and $C_{S_1,S_2}$ not violating the coprimality condition. 
\item For $(S_1,S_2) = (\emptyset,\emptyset)$ let $\sigma$ be a uniformly random matching between $B_{\emptyset,\emptyset}\cup R_B$ and $C_{\emptyset,\emptyset}\cup R_C$. 
\item Output $(\sigma, R_B, R_C)$.
\end{itemize}
Note that this fully specifies $\sigma$ since if $S_1\cap S_2\neq\emptyset$ we have $x\in B_{S_1}$ implies $\phi_1(x)\neq S_2$ and similar for $\phi_2$. Furthermore, given the output $(\sigma,R_B,R_C)$ we can mostly reconstruct the procedure. In particular, for $|S_1|,|S_2|<k$, $B_{S_1,S_2}$ is the set of values in $B_{S_1}$ mapped to $C_{S_2}$ by $\sigma$ for all $(S_1,S_2)\neq(\emptyset,\emptyset)$, while $B_{\emptyset,\emptyset}$ is the set of values in $B_\emptyset\setminus R_B$ mapped to $C_\emptyset\cup R_C$. This determines $\sigma_1$ up to right-multiplication by some $\xi_1$ which preserves the sets $B_{S_1,S_2}$ for all $|S_1|,|S_2|<k$ as well as $R_B$, and similar for $\sigma_2$.

Therefore, if $X$ is the number of possible outputs $(\sigma,R_B,R_C)$ where $\sigma$ is coprime and $Y$ is the number of ways to run the above procedure where the output is coprime, we have
\begin{equation}\label{eq:reduce-to-matching}
Y\le X\cdot\beta_\ast^1!\beta_\ast^2!\prod_{\substack{|S_1|,|S_2|<k\\S_1\cap S_2=\emptyset}}\beta_{S_1,S_2}!^2\le C_0(n)\cdot n^{\beta_\ast^1+\beta_\ast^2}\prod_{\substack{|S_1|,|S_2|<k\\S_1\cap S_2=\emptyset}}\beta_{S_1,S_2}!^2.
\end{equation}
Now we provide a lower bound on $Y$.

Note that by \cref{lem:prime-division} and the Chernoff bound (\cref{lem:chernoff}), with probability at least $3/4$ we have that for all disjoint $|S_1|,|S_2|<k$ the coprimality bipartite graph between parts $B_{S_1,S_2}$ and $C_{S_1,S_2}$ for $(S_1,S_2)\neq(\emptyset,\emptyset)$ has at most $4W^{-1/2}\beta_{S_1,S_2}$ edges missing from each vertex. Furthermore with probability $3/4$ the non-coprimality bipartite graph between $B_{\emptyset,\emptyset}\cup R_B$ and $C_{\emptyset,\emptyset}\cup R_C$ has at most $4W^{-1/2}\beta_{\emptyset,\emptyset}+|R_B|$ edges missing from each vertex. Using that $|\beta_{\emptyset, \emptyset}|\ge n\exp(-2\log\log n)$ and $|R_B| = \beta_\ast^1\le 2\alpha^{1/2}n$, we see that the maximum degree of the complement bipartite graph is bounded by $(4W^{-1/2}+\alpha^{1/3})\beta_{\emptyset,\emptyset}$, say.

Therefore applying \cref{lem:matchings} shows that there are
\begin{equation}\label{eq:Y-lower-bound}
Y\ge (3/4)^2\beta_\ast^1!\beta_\ast^2!\prod_{|S|<k}|B_S|!\prod_{|S|<k}|C_S|!\cdot\prod_{\substack{|S_1|,|S_2|<k\\S_1\cap S_2=\emptyset}}(\beta_{S_1,S_2}/e)^{\beta_{S_1,S_2}}e^{-4(4W^{-1/2}+\alpha^{1/3})\beta_{S_1,S_2}}
\end{equation}
possible valid ways to run the procedure. Now, given \cref{eq:reduce-to-matching,eq:Y-lower-bound}, a nearly identical calculation as in \cref{sec:upper-bound} demonstrates
\[C_0(n)\ge\exp(-n\alpha^{1/4})n!\bigg(\prod_{p > 2}\frac{(p-1)^{2(1-1/p)}}{p\cdot (p-2)^{(1-2/p)}}\bigg)^n.\]
Finally, using \cref{lem:even-reduction} implies the result.

\bibliographystyle{amsplain0.bst}
\bibliography{main.bib}

\end{document}